\documentclass[a4paper,10pt,twoside]{article}
\usepackage{pst-all}
\usepackage{graphicx}
\usepackage{url}
\usepackage{fancyhdr}
\usepackage{amsmath}
\usepackage{epsfig} 
\usepackage{psfrag} 
\usepackage{amsmath,amsfonts,amssymb,amsthm} 
\usepackage{bbm} 
\newtheorem{Theorem}{Theorem}[section] 
\newtheorem{Corollary}[Theorem]{Corollary} 
\newtheorem{Lemma}[Theorem]{Lemma}

\theoremstyle{definition} 

\newcommand{\R}{\mathbbm{R}}
\newcommand{\st}{\mbox{s.t.}}
\newcommand{\MI}{\mathcal{I}}
\newcommand{\MO}{\mathcal{O}}
\DeclareMathOperator{\conv}{conv}
\DeclareMathOperator{\cone}{cone}
\DeclareMathOperator{\Tree}{P_{Tree}}
\DeclareMathOperator{\GHTree}{P_{GH-Tree}}
\DeclareMathOperator{\Proj}{Proj}
\DeclareMathOperator{\dist}{dist}

\title{An extension of disjunctive programming and its impact for compact tree formulations}
\author{R\"udiger Stephan}
\date{}
\begin{document}

\maketitle

\begin{abstract}
In the 1970's,
Balas~\cite{Balas79, Balas10} introduced the concept of disjunctive programming,
which is optimization over unions of polyhedra.
One main result of his theory is that, given linear descriptions for each of the polyhedra to be taken in the union, one can easily derive an extended formulation of the convex hull of the union of these polyhedra. In this paper, we give a generalization of this result by
extending the polyhedral structure of the variables coupling the polyhedra taken in the union.
Using this generalized concept, we derive polynomial size linear programming formulations
({\em compact formulations}) for a well-known spanning tree approximation of Steiner trees,
for Gomory-Hu trees, and, as a consequence, of the minimum $T$-cut problem
(but not for the associated $T$-cut polyhedron). Recently, Kaibel and Loos~\cite{KL10}
introduced a more involved framework called {\em polyhedral branching systems} to derive extended formulations. The most parts of our model can be expressed in terms of their framework. The value of our model can be seen in the fact that it completes their framework by an interesting algorithmic aspect.
\end{abstract}

\section{Introduction}

Let $Q:= \{(x,y) \in \R^p \times \R^q\,|\, Ax+By \geq c\}$ be a polyhedron. The {\em projection of $Q$ onto the $x$-space} is the polyhedron
$$ \Proj_x(Q):=\{x \in \R^p\,|\, \exists y \in \R^q: (x,y) \in Q\}.$$
Conversely, $Q$ is said to be an {\em extension} of the projected polyhedron $P:= \Proj_x(Q)$,
and the system $Ax+By \geq c$ is an {\em extended formulation} for $P$.
The extended formulation $Ax+By\geq c$ is {\em compact} if $q$, the number of rows,
and the input length of each entry of the inequality system is polynomial in $p$.

In the 1970's,
Balas~\cite{Balas79, Balas10} introduced the concept of disjunctive programming,
which is optimization over unions of polyhedra. Below we restate a well-known result saying that, given linear descriptions for each of the polyhedra to be taken in the union, one can easily derive an extended formulation of the convex hull of the union of these polyhedra.
Recently, Kaibel and Loos~\cite{KL10} introduced a powerful framework called {\em polyhedral branching systems} that generalizes Balas' result as well as the framework of Martin, Rardin, and Campbell~\cite{MRC88} to derive extended formulations from dynamic programming algorithms for combinatorial optimization problems.

In this paper, we consider a  generalization of the concept of disjunctive programming
whose polyhedral aspects are covered by the framework of Kaibel and Loos~\cite{KL10}.
The motivation of our model lies in the algorithmic interpretation of disjunctive programming as a two-level approach to solve a linear optimization problem. Since the generalization is straightforward, while the framework of polyhedral branching systems is quite involved (at least, it would need some space to explain it), we pass on a description of that framework and directly start our considerations with disjunctive programming, its interpretation as two-level optimization model, and its consequences for extended formulations.
  
Given a finite collection of polyhedra
$P^i$, $i \in \MI$, where $\MI$ is a finite index set, a {\em disjunctive program} is a mathematical program of the form
\begin{equation} \label{eq:dpmax}
\begin{array}{cr@{\,}l}
\max & w^Tx\\
\mbox{s.t.} & x & \in \bigcup\limits_{i\in \MI}P^i.
\end{array}
\end{equation}

By a well-known result of Balas~\cite{Balas05}, 
given complete linear descriptions of each of the polyhedra $P^i$ to be taken in the union,
one can describe the convex combination of the union of the polyhedra by an extended
formulation.

\begin{Theorem}[Balas~\cite{Balas05}] \label{T:dp1}
Given polyhedra
$$P^i = \{x \in \R^n \,|\, A^i x \geq b^i\} \, = \, \conv V^i + \cone R^i,\; i \in \MI,$$
 the following system:
\begin{equation} \label{eq:dp}
\begin{array}{r@{\,}l@{\hspace{1.0cm}}l}
 x - \sum\limits_{i \in \MI} x^i &=0,\\[0.8em]
A^i x^i -\lambda_i b^i &\geq 0, & i \in \MI,\\
\sum\limits_{i\in \MI} \lambda_i  &= 1,\\
\lambda_i  &\geq  0,  & i \in \MI
\end{array}
\end{equation}
provides an extended formulation for the polyhedron $$P_\MI := \conv \bigcup\limits_{i\in \MI} V^i + \cone \bigcup\limits_{i\in \MI} R^i.$$
In particular, denoting by $P$ the set of vectors $(x, \{x^i,\lambda_i\}_{i \in \MI})$
satisfying~\eqref{eq:dp},
\begin{itemize}
\item[(i)] if $x^\star$ is a vertex of $P_\MI$, then $(\bar{x}, \{\bar{x}^i,\bar{\lambda}_i\}_{i \in \MI})$
is a vertex of $P$, with $\bar{x}=x^\star$, $(\bar{x}^k,\bar{\lambda}^k)=(x^\star,1)$
for some $k \in \MI$, and $(\bar{x}^i,\bar{\lambda}_i)=(0,0)$ for $i \in \MI \setminus \{k\}$;
\item[(ii)] if  $(\bar{x}, \{\bar{x}^i,\bar{\lambda}_i\}_{i \in \MI})$ is a vertex of $P$, then
$(\bar{x}^k,\bar{\lambda}^k)=(\bar{x},1)$ for some $k \in \MI$, $(\bar{x}^i,\bar{\lambda}_i)=(0,0)$ for
$i \in \MI \setminus \{k\}$, and $\bar{x}$ is a vertex of $P_\MI$. \hfill $\Box$
\end{itemize}
\end{Theorem}

By Theorem~\ref{T:dp1}, the disjunctive program~\eqref{eq:dpmax} can be solved by solving the linear program
$\max \{w^Tx\,|\, (x, \{x^i,\lambda_i\}_{i \in \MI}) \mbox{ satisfies~\eqref{eq:dp}}\}$,
provided we are given linear descriptions of each polyhedron $P^i$ as required.

From an algorithmic viewpoint, to solve~\eqref{eq:dpmax}, one usually would
compute an optimal solution of each subproblem $\max \{w^Tx\,|\,x \in P^i\}$,
and then one would choose the best (or a best) among them. This two-level approach
is reflected in the extended formulation. For simplicity, let us assume that,
in the first phase,
for each subproblem $\max \{w^Tx\,|\,x \in P^i\}$ an optimal solution $\bar{x}^i$ exists.
Let $\lambda \in \R^\MI$ be the (variable) vector whose components are the $\lambda_i$.
Then, one defines a vector $\bar{w}$ by $\bar{w}_i := w^T\bar{x}^i$, $i\in \MI$, and
solves, in the second phase, the linear program $\max \{\bar{w}^T\lambda \,|\, \lambda \in \Delta\}$
over the simplex
$$\Delta := \{\lambda \in \R^\MI : \sum_{i \in \MI} \lambda_i=1, \, \lambda_i \geq 0\;\forall\, i \in \MI\}.$$

Given linear programs $\max \{w^Tx\,|\,x \in P^i\}$, it is, however, not always intended to optimize over the union of these polyhedra, but sometimes the subproblems are part of a more complex
optimization problem.

As an example, let us consider (the polyhedral version of)
the minimum spanning tree problem over the metric closure of a weighted graph.
This problem has relevance for the approximation of the Steiner tree problem,
which will be discussed in more detail in Section~\ref{Sec:2}.
We denote the node and edge set of a graph $G$ by $V(G)$ and $E(G)$, respectively.
Given a graph $G$ with a nonnegative edge weighting $w:E(G) \to \R_+$,
the {\em metric closure} of $(G,w)$ is the pair $(K,\bar{w})$, where $K$
is a complete graph on node set $V(G)$ and $\bar{w}(e)$, for $e=\{s,t\}$, is defined to be the length
of a shortest path connecting $s$ and $t$ in $G$ w.r.t. to $w$ if there is any such path,
and otherwise $\bar{w}(e):= + \infty$.
The aim is now to find a spanning tree $T$ of $K$
minimizing $\bar{w}(E(T)) := \sum_{e\in E(T)} \bar{w}(e)$.
The approach obviously consists of a two-level model. In the first step,
we solve a so-called {\em all-pairs shortest path problem}, and in the second step,
a minimum spanning tree problem whose input data are given by the output data
of the first problem.
For each of the two problems, there are known several linear programming formulations and,
among these, even compact formulations. The question now arises whether or not these
formulations can be brought together to provide a linear programming formulation for the entire
problem. The answer to this question is surprisingly quite easy and can be given using
a modification of~\eqref{eq:dp}. Let $\MI := E(K)$ be given by the edge set of $K$, and
let $\min \{\sum w(e)x_e\,|\, A^ex \geq b^e\}$ be a linear programming formulation of a shortest $s,t$-path
problem in $G$ for each edge $e=\{s,t\} \in E(K)=\MI$.
Then, replacing $\Delta$ by a linear characterization
$\Pi := \{\lambda \in \R^{E(K)}\,:\, C\lambda \ge d\}$
of the spanning tree polytope, we propose the following model:
\begin{equation*}
\begin{array}{lr@{\,}l@{\hspace{1.0cm}}l}
\multicolumn{4}{l}{\min \;\sum\limits_{e \in E(G)}w_ex_e}\\[0.8em]
 \mbox{s.t.}& x - \sum\limits_{e \in E(K)} x^e &=0\\[0.8em]
&A^e x^e -\lambda_e b^e &\geq 0, & e \in E(K),\\
&C\lambda & \ge d.
\end{array}
\end{equation*}
In any optimal solution $(\bar{x}, \{\bar{x}^e,\bar{\lambda}_e\}_{e \in E(K)})$,
$\bar{\lambda}$ is then a convex combination of the characteristic vectors of minimum spanning
trees w.r.t.~$\bar{w}:E(K)\to\R_+$.

The aim of the remainder of this section is to prove the correctness of this approach in general.

For any polyhedron $P= \{x \in\R^n\,|\,Ax \geq b\}$ and any $\alpha\in\R_+$,
let $P(\alpha) := \{x \in\R^n\,|\,Ax \geq \alpha b\}$. This implies
that
$$
\begin{array}{cr@{\;}l}
& P &=\conv \{v_1,v_2,\ldots,v_k\}+\cone \{r_1,r_2,\ldots,r_m\}\\
\Leftrightarrow&P(\alpha)&=\conv \{\alpha v_1,\alpha v_2,\ldots,\alpha v_k\}+\cone \{r_1,r_2,\ldots,r_m\}.
\end{array}
$$

\begin{Theorem} \label{T:dp2}
Given pointed polyhedra
$P^i = \{x \in \R^n \,|\, A^i x \geq b^i\}$, $i \in \MI$ and
a 0/1-polytope $\Pi = \{\lambda \in \R^\MI\,|\, C\lambda \ge d\}$, that is, $\Pi = \conv V$ for some
$V \subseteq \{0,1\}^\MI$.
\begin{itemize}
\item[(i)] $(\bar{x}, \{\bar{x}^i\}_{i \in \MI}, \bar{\lambda})$ is a vertex of the polyhedron $Q$ defined as the set of vectors $(x, \{x^i\}_{i \in \MI}, \lambda)$ satisfying
\begin{equation} \label{eq:dpextend}
\begin{array}{r@{\,}l@{\hspace{1.0cm}}l}
 x - \sum\limits_{i \in \MI} x^i &=0,\\[0.8em]
A^i x^i -\lambda_i b^i &\geq 0, & i \in \MI,\\
C\lambda &\ge d
\end{array}
\end{equation}
if and only if $\bar{\lambda}$ is a vertex of $\Pi$, and for each $i \in \MI$,
$\bar{x}^i$ is a vertex of $P^i$ if $\lambda_i=1$ and $\bar{x}^i=0$ otherwise.
\item[(ii)] For any $w \in \R^n$, consider the linear programs
\begin{equation}\label{eq:dp4}
\max \{w^Tx\,|\,(x, \{x^i\}_{i \in \MI}, \lambda) \in Q\},
\end{equation}
\begin{equation}\label{eq:dp5} \tag{5i}
\bar{w}_i:=\max \{w^Tx\,|\,x \in P^i\}
\end{equation}
for $i \in \MI$,
and
\setcounter{equation}{5}
\begin{equation}\label{eq:dp6}
\max \{\bar{w}^T\lambda\,|\, \lambda \in \Pi\}.
\end{equation}
Then, \eqref{eq:dp4} is unbounded if and only if $\bar{w}_i=\infty$ for some $i\in \MI$.
Next, let~\eqref{eq:dp4} be bounded.
Then, $(\bar{x}, \{\bar{x}^i\}_{i \in \MI}, \bar{\lambda})$ is an optimal solution of~\eqref{eq:dp4}
if and only if $\bar{\lambda}_i^{-1}\bar{x}^i$ is an optimal solution of~\eqref{eq:dp5}
for each $i \in \MI$ with $\bar{\lambda}_i>0$ and
$\bar{\lambda}$ is an optimal solution of~\eqref{eq:dp6}.
\end{itemize}
\end{Theorem}

\begin{proof}
(i) To show the necessity, let $(\bar{x}, \{\bar{x}^i\}_{i \in \MI}, \bar{\lambda})$ be a vertex of $Q$.
First suppose, for the sake of contradiction, that $\bar{\lambda}$ is not a vertex of $\Pi$.
Then, there are $\mu,\,\nu \in \Pi$ and $0 < \alpha < 1$ such that $\bar{\lambda}=\alpha\mu+(1-\alpha)\nu$.
Let $\MI'$ be the set of indices $i \in \MI$ with $\bar{\lambda}_i=0$. This implies that
$\mu_i=\nu_i=0$ for $i \in \MI'$.
Define now vectors $y^i:=z^i:=\bar{x}^i$ for $i\in \MI'$,
$y^i:=\frac{\mu_i}{\bar{\lambda}_i}\bar{x}^i$, $z^i:=\frac{\nu_i}{\bar{\lambda}_i}\bar{x}^i$ for $i\in \MI \setminus \MI'$, as well as $y:=\sum_{i \in \MI}y^i$ and
$z:=\sum_{i \in \MI}z^i$. Then, one easily verifies that, on the one hand,
$$ \alpha(y,\{y^i\}_{i\in \MI},\mu)+(1-\alpha)(z,\{z^i\}_{i\in \MI},\nu) = (\bar{x}, \{\bar{x}^i\}_{i \in \MI}, \bar{\lambda}),$$
and on the other hand, $(y,\{y^i\}_{i\in \MI},\mu),\,(z,\{z^i\}_{i\in \MI},\nu) \in Q$.
Thus, $(\bar{x}, \{\bar{x}^i\}_{i \in \MI}, \bar{\lambda})$ is not a vertex, a contradiction.

Next, suppose that $(\bar{x}, \{\bar{x}^i\}_{i \in \MI}, \bar{\lambda})$ is a vertex of $Q$ and $\bar{\lambda}$ a vertex
of $\Pi$. The latter implies that $\bar{\lambda} \in \{0,1\}^\MI$. Assume now that
$\bar{x}^j$ is not a vertex of $P^j$ for some $j$ with $\bar{\lambda}_j=1$.
Then, $\bar{x}^j$ is the convex combination of two vectors $y^j,\,z^j \in P^j$.
This, in turn, implies that $(\bar{x}, \{\bar{x}^i\}_{i \in \MI}, \bar{\lambda})$ is a convex combination
of the two vectors obtained from $(\bar{x}, \{\bar{x}^i\}_{i \in \MI}, \bar{\lambda})$ by
replacing the vector $\bar{x}^j$ by $y^j$ and $z^j$.
Moreover, assuming that $\bar{x}^j \neq 0$ for some $j$ with $\bar{\lambda}_j=0$, we see that $\bar{x}^j$
is a ray of the cone $\{x \in \R^n\,|\, A^jx \geq 0\}$. This immediately implies that, also in this case,
$(\bar{x}, \{\bar{x}^i\}_{i \in \MI}, \bar{\lambda})$ is not a vertex of $Q$, a contradiction.

To show the sufficiency, suppose that
$(\bar{x}, \{\bar{x}^i\}_{i \in \MI}, \bar{\lambda})$ is not a vertex of $Q$.
Then, there are two different vectors $(y,\{y^i\}_{i\in \MI},\mu),\,(z,\{z^i\}_{i\in \MI},\nu) \in Q$
and $0<\alpha<1$ such that
$$\alpha(y,\{y^i\}_{i\in \MI},\mu)+(1-\alpha)(z,\{z^i\}_{i\in \MI},\nu)=(\bar{x}, \{\bar{x}^i\}_{i \in \MI}, \bar{\lambda}).$$
First, assume that $\mu\neq\nu$. Since both vectors are in $\Pi$, this immediately implies that
$\bar{\lambda}$ is a convex combination of $\mu$ and $\nu$, a contradiction. Consequently,
we may assume that $\mu=\nu=\bar{\lambda} \in \{0,1\}^\MI$.
However, since $(y,\{y^i\}_{i\in \MI},\mu)$ and $(z,\{z^i\}_{i\in \MI},\nu)$ have to be distinct,
and since the vectors $y$ and $z$ are just sums of the vectors $y^i$ and $z^i$, respectively,
we conclude that $y^j \neq z^j$ for some $j\in \MI$. If $\bar{\lambda}_j=1$, it follows that
$\bar{x}^j$ was not a vertex of $P^i$. Finally, if $\bar{\lambda}_j=0$, then
$\alpha y^j+(1-\alpha)z^j=\bar{x}^j \neq 0$ or $\alpha y^j+(1-\alpha)z^j= 0$, and hence $P^j$ is not a pointed polyhedron. In either case, this yields a contradiction.

\vspace{\baselineskip}
(ii) If~\eqref{eq:dp4} is unbounded, then there exists a ray $(r,\{r^j\}_{j\in \MI},\rho) \neq 0$ of $P$
with $w^Tr>0$, which implies $r \neq 0$.
Since $\Pi$ is a polytope, it follows that $\rho=0$. Moreover, since $r$ is the sum of the vectors $r^j$,
$w^Tr^i>0$ for at least one $i \in \MI$. Since, by definition, $A^ir^i\geq 0$, this implies that
$r^i$ is a ray of $P^i$ and~\eqref{eq:dp5} is unbounded. Conversely, if~\eqref{eq:dp5} is unbounded for some $i \in \MI$, there exists a ray $\tilde{r}$ of $P^i$ with $w^T\tilde{r}>0$. Define $(r,\{r^j\}_{j\in \MI},\rho)$
by $r:=r^i:=\tilde{r}$, $r^j :=0$ for $j \in \MI \setminus \{i\}$, and $\rho:=0$.
Then, we conclude that $(r,\{r^j\}_{j\in \MI},\rho)$ is a ray of $Q$ and $w^Tr>0$. Therefore,
\eqref{eq:dp4} is unbounded.

Next, suppose that~\eqref{eq:dp4} is bounded, which means that~\eqref{eq:dp5} is bounded for each $i\in \MI$.
This, in turn, justifies the definition of $\bar{w}$.

For any $i\in \MI$ and any $\alpha\in\R_+$,
$x^\star$ is an optimal solution of $\max \{w^Tx\,|\,x \in P^i\}$ if and only if
$\alpha x^\star$ is one of $\max \{w^Tx\,|\,x \in P^i(\alpha)\}$.
Now let $(\bar{x}, \{\bar{x}^j\}_{j \in \MI}, \bar{\lambda})$ be an optimal solution of~\eqref{eq:dp4}.
Assume that $w^Tx^\star>w^T(\bar{\lambda}_i^{-1}\bar{x}^i)$ for some $x^\star \in P^i$, with $i\in \MI$
and $\bar{\lambda}_i>0$.
Then, $(x, \{x^j\}_{j \in \MI}, \bar{\lambda}) \in Q$, where $x^i:=\bar{\lambda}_ix^\star$,
$x^j:=\bar{x}^j$ for $j\in \MI\setminus \{i\}$, and $x:=\sum_{i \in \MI}x^i$. Moreover, $w^Tx>w^T\bar{x}$,
a contradiction.
Next, for any $\lambda \in \Pi$ and any optimal solutions $\tilde{x}^i$ of~\eqref{eq:dp5} for each $i\in \MI$,
we derive that $(x, \{x^j\}_{j \in \MI}, \lambda) \in Q$, where $x:=\sum_{i\in \MI}x^i$ and $x^i:=\lambda_i\tilde{x}^i$
for $i\in \MI$. Thus, $w^Tx \le w^T\bar{x}$. Since
$$
\begin{array}{cr@{\;}l}
&w^Tx &\le w^T\bar{x}\\
\Leftrightarrow& \sum_{i \in \MI} w^Tx^i&\leq \sum_{i \in \MI} w^T\bar{x}^i\\
\Leftrightarrow&\sum_{i \in \MI} w^T(\lambda_i\tilde{x}^i) &\leq \sum_{i \in \MI} w^T\bar{x}^i\\
\Leftrightarrow&\sum_{i \in \MI} \lambda_iw^T\tilde{x}^i &\leq \sum_{i \in \MI} \bar{\lambda}_iw^T\tilde{x}^i\\
\Leftrightarrow&\sum_{i \in \MI} \lambda_i\bar{w}_i&\leq \sum_{i \in \MI} \bar{\lambda}_i\bar{w}_i\\
\Leftrightarrow&\bar{w}^T\lambda &\leq \bar{w}^T\bar{\lambda},
\end{array}
$$
it follows that $\bar{\lambda}$ is an optimal solution of~\eqref{eq:dp6}.

For the same reasons, if $\bar{x}^i$ is an optimal solution of~\eqref{eq:dp5} for each $i \in \MI$
and $\bar{\lambda}$ is one of~\eqref{eq:dp6}, then $(\bar{x}, \{\bar{x}^j\}_{j \in \MI}, \bar{\lambda})$ is optimal
for~\eqref{eq:dp4}, where $\bar{x}:=\sum_{i\in \MI}\bar{x}^i$.
\end{proof}

Theorem~\ref{T:dp2} can be generalized in several ways. Of course,
it can be easily extended for the case that $\Pi$ is not a 0/1-polytope but any other polytope contained in
$\R^\MI_+$. More important for the following applications is the case in which the polyhedra $P^i$ and $\Pi$ are given themselves by extended formulations. The following theorem only generalizes those results of
Theorem~\ref{T:dp2} that are relevant for the following applications.

\begin{Theorem} \label{T:dp3}
Given extensions $\Theta:= \{(\lambda,\mu) \in \R^\MI\times \R^q \,|\, C\lambda+D\mu \geq d\}$ of a 0/1-polytope
$\Pi \subseteq \R^\MI$ as well as
$Q^i = \{(x,y) \in \R^n\times \R^{p_i} \,|\, A^i x +B^iy \geq b^i\}$ of pointed polyhedra
$P^i \subseteq \R^n$, $i \in \MI$. Moreover, for any $w \in \R^n$, consider the linear programs
\begin{equation}\label{eq:dp9}
\begin{array}{lr@{\;}l@{\hspace{1cm}}l}
\multicolumn{4}{l}{\max \;w^T\sum\limits_{i\in\MI}x^i}\\[0.8em]
\st  & A^i x^i + B^iy^i -\lambda_i b^i &\geq 0, & i \in \MI,\\
     & C\lambda+D\mu &\geq d,
\end{array}
\end{equation}
\begin{equation}\label{eq:dp7} \tag{8i}
\bar{w}_i:=\max \{w^Tx\,|\, (x,y) \in Q^i\}
\end{equation}
for $i \in \MI$, and
\setcounter{equation}{8}
\begin{equation}\label{eq:dp8}
\max \{\bar{w}^T\lambda\,|\, (\lambda,\mu) \in \Theta\}.
\end{equation}
Then, ~\eqref{eq:dp9} is unbounded if and only if $\bar{w}_i=\infty$ for some $i \in \MI$.
Moreover, in case that $\bar{w}_i<\infty$ for all $i \in \MI$,
$(\{\bar{x}^i,\bar{y}^i\}_{i \in \MI},\bar{\lambda}, \bar{\mu})$
is an optimal solution of~\eqref{eq:dp9} if and only if
$(\bar{x}^i,\bar{y}^i) \in Q^i$, $i\in\MI$, are optimal solutions of~\eqref{eq:dp7}
and $(\bar{\lambda},\bar{\mu}) \in \Theta$ is an optimal solution of~\eqref{eq:dp8}.
\hfill$\Box$
\end{Theorem}

%\begin{Remark} \label{Rem:dp10}
%If the goal in a two-level model is not to maximize but to minimize $\bar{w}$ over $\Pi$,
%then one can introduce complementary variables $\nu_i=1-\lambda_i$, $i \in \MI$, and adjusts
%the linear program~\eqref{eq:dp9}:
%\begin{equation}\label{eq:dp10}
%\begin{array}{lr@{\;}l@{\hspace{1cm}}l}
%\multicolumn{4}{l}{\max \;w^T\sum\limits_{i\in\MI}x^i}\\[0.8em]
%\st  & A^i x^i + B^iy^i -\nu_i b^i &\geq 0, & i \in \MI,\\
%     & C\lambda+D\mu &\geq d,\\
%     & \lambda+\nu &=\mathbbm{1},
%\end{array}
%\end{equation}
%where $\mathbbm{1}$ is the vector of all ones.
%With $\Pi' := \{(\lambda,\mu,\nu)\,|\,C\lambda+D\mu \geq d,\,\lambda+\nu =\mathbbm{1}\}$,
%we conclude that for any $(\lambda,\mu,\nu) \in \Pi'$, $\bar{w}^T\lambda$ is minimal if and only if
%$\bar{w}^T\nu$ is maximal.
%\end{Remark}

An application of Theorem~\ref{T:dp3} can be found in Pochet and Wolsey~\cite{PW93}
for a variant of the classical lot sizing problem with constant production capacities
in which the capacity in each period is an integer multiple of a basic capacity unit.

In the following sections, we give two further applications
of our model. Since both applications are based on compact formulations for spanning trees, we recall two well-known spanning tree formulations.

Given a graph $G$, the {\em spanning tree polytope} $\Tree(G)$ is the convex hull of the characteristic vectors $\chi^{E(T)} \in \R^{E(G)}$
of spanning trees $T \subseteq G$.
Recall that the characteristic vector $\chi^F$ of any edge set $F \subseteq E(G)$ is a 0/1-vector with $\chi^F_e=1$ if and only if $e \in F$.
As it is well known, the spanning tree polytope is the set of all
$\lambda \in \R^{E(G)}$ satisfying the nonnegativity constraints $\lambda_e \geq 0$, $e \in E$,
the equation $\lambda(E(G))=|V(G)|-1$, and the inequalities
\begin{align*}
\lambda(E(G[S])) \leq |S|-1 &&\mbox{for all $S \subset V(G), \, 2 \leq |S| \leq |V(G)|-1$}.
\end{align*}
Here, $G[S]$ denotes the subgraph of $G$ induced by $S \subseteq V$,
and for any $F \subseteq E(G)$, $\lambda(F):= \sum_{e \in F} \lambda_e$.
This linear description of  $\Tree(G)$ has an exponential number of inequalities.

According to~\cite{CCZ09, KMartin91, Yannakakis91}, a compact formulation for $\Tree(G)$ can be given as follows.
For each edge $e \in E(G)$, we introduce a variable $\lambda_e$, and for each
for each edge $e \in E(G)$, each node $u \in e$, and each node $v \in V(G)$, we
introduce a variable $\mu_{e,u,v}$.
In the extended formulation, the edge set of a spanning tree $T$ of $G$
will be represented by the vector $(\lambda^T,\mu^T)$ with $\lambda^T=\chi^{E(T)}$
and $\mu^T_{e,u,v}=1$ if and only if $e \in E(T)$ and $v$ belongs to the same component of $u$ in $T-e$.
Then, $\Tree(G)$ is the projection of the polyhedron $\Pi$ defined as the set of all $(\lambda,\mu)$ satisfying
\begin{equation} \label{eq:tree}
\begin{array}{r@{\;}lr}
\lambda(E) &=|V|-1,\\
\mu_{\{u,v\},u,v}=\mu_{\{u,v\},v,u} &=0& \mbox{for all $\{u,v\} \in E$,}\\
\lambda_e - \sum\limits_{v \in e}\mu_{e,v,w}& = 0 & \mbox{for all $e\in E,\,w \in V$,}\\
\lambda_{\{u,v\}} + \sum\limits_{w \in V \setminus \{u,v\}} \mu_{\{u,w\},w,v} & =1 &
\mbox{for all $\{u,v\} \in E$,}\\
\lambda_e&\ge 0& \mbox{for all $e \in E$,}\\
\mu_{e,v,w} &\ge 0& \mbox{for all $e \in E,v \in e,w \in V$,}
\end{array}
\end{equation}
onto the $\lambda$-space. Here, $V:=V(G)$ and $E:=E(G)$.

For later reference, we need the following result which can be found in~\cite{KMartin91}.
\begin{Lemma} \label{L:tree}
$\Pi$ is an integer polyhedron. Moreover, $(\lambda,\mu) \in \Pi$ is a vertex of $\Pi$ if and only if $\lambda$ is the characteristic vector of a spanning tree $T$ of $G'$, and
for each $e \in E(T)$ and each $u \in e$, the vector $(\mu_{e,u,v})_{v \in V(G')}$ is the characteristic vector
of the component of $T-e$ containing $u$, and for each $e \in E(G')\setminus E(T)$, $\mu_{e,u,v}=0$ for all
$u \in e,\,v \in V$. \hfill $\Box$
\end{Lemma}

The following directed formulation is due to Wong~\cite{Wong84} and Maculan~\cite{Maculan86}. Let $D$ be a directed graph, with node set $V(D)$ and arc set $A(D)$. For a fixed node $r \in V(D)$, the {\em $r$-arborescence polytope},
that is, the convex hull of the characteristic vectors of $r$-arborescences, is
the projection of the system
\begin{equation} \label{eq:directedtree}
\begin{array}{r@{\;}lr}
\nu(A) &=|V|-1,\\
\nu_a & \geq \sigma^w_a & \mbox{for all $a \in A, w \in V \setminus \{r\}$,}\\
\sigma^w(\delta^+(r))&=1 &\mbox{for all $w \in V \setminus \{r\}$,}\\
\sigma^w(\delta^+(v))-\sigma^w(\delta^-(v))&=0&\mbox{for all $v,w \in V\setminus \{r\}, v \neq w$,}\\
\sigma^w_a & \ge 0& \mbox{for all $a \in A, w \in V \setminus \{r\}$}
\end{array}
\end{equation}
onto the $\nu$-space, where $V:= V(D)$, $A:= A(D)$,
and for any $v \in V$, $\delta^+(v)$ and $\delta^-(v)$ are the set of arcs leaving
and entering $v$, respectively.
This formulation is motivated by the fact that an arborescence with root $r$ contains an $(r,w)$-path for each node $w \neq r$.
Let now $G$ be a graph and $D$ the digraph obtained from $G$ by replacing each edge $\{u,v\}$ of $G$ by the arcs $(u,v)$ and $(v,u)$. Then,
$\Tree(G)$ is the projection of~\eqref{eq:directedtree} extended by the inequalities
\begin{align}
\lambda_{\{u,v\}}& = \nu_{uv}+\nu_{vu} & \mbox{for all $\{u,v\} \in E(G)$}  \label{eq:100}
\end{align}
onto the $\lambda$-space.

The remainder of this paper is organized as follows.
In Section~\ref{Sec:2}, we provide a compact linear programming formulation for a well-known spanning tree approximation of the minimum Steiner tree problem.
In Section~\ref{Sec:3}, we present a compact formulation
for Gomory-Hu trees.  Based on this formulation, we derive one for the minimum $T$-cut problem
whose inequalities depend, however, on the objective function.
Moreover, we discuss the relevance of these formulations for finding a
compact formulation for the perfect matching polytope.
In Section~\ref{Sec:4}, we briefly summarize our findings and point out some interesting open questions.

\section{A compact formulation for the approximation of the Steiner tree problem} \label{Sec:2}

Let $G$ be an undirected graph and $S \subseteq V(G)$. A {\em Steiner tree} for $S$ in $G$ is a
tree $T \subseteq G$ whose node set $V(T)$ contains $S$. Given a cost function $c:E(G)\to\R_+$,
in the {\em Steiner tree problem} for $(G,c,S)$,
one wants to find a Steiner tree $T \subseteq G$ minimizing $c(E(T))$.

Let $(K,\bar{c})$ be the metric closure of $(G,c)$.
By a well-known result of Gilbert and Pollak~\cite{GP68},
if $T$ is a minimum Steiner tree for $S$, and $M$ is a minimum spanning tree in $K[S]$
w.r.t. $\bar{c}$, then $\bar{c}(E(M)) \leq 2c(E(T))$,
where $K[S]$ denotes the subgraph of $K$ induced by $S$.

For the all-pairs shortest path problem (for the computation of the metric closure) as well as for
the minimum spanning tree problem there exist compact formulations. Hence, using Theorem~\ref{T:dp3},
it is easy to derive a compact formulation approximating the Steiner problem.

For any digraph $D$, we denote by $V(D)$ and $A(D)$ the node and arc set of $D$, respectively.
Let $G$ be a graph and $c:E(G)\to\R_+$.
In what follows, let $D$ be the digraph obtained from $G$ by replacing each edge $e=\{v,w\} \in E(G)$ by the arcs
$(v,w)$ and $(w,v)$. Define $\tilde{c}:A(D) \to \R$ by $\tilde{c}((v,w)):=c(\{v,w\})$ for $(v,w) \in A(D)$.
Then, a shortest $s,t$-path problem in $(G,c)$ can be modeled as a minimum cost flow problem in
$(D,\tilde{c})$.
Consequently, for each edge $e=\{s,t\} \in E(K[S])$,
\begin{equation} \label{eq:shortestpath} \tag{13e}
\begin{array}{lr@{\;}l@{\hspace{1cm}}r}
\multicolumn{4}{l}{\min \; \sum\limits_{a \in A(D)} \tilde{c}(a)x^e_a}\\[0.8em]
\st  & x^e(\delta^+(s)) - x^e(\delta^-(s))  &=1,\\
     & x^e(\delta^+(v)) - x^e(\delta^-(v))  &=0 &\mbox{for all $v \in V(D) \setminus \{s,t\}$},\\
     & 0 \le x^e_a & \le 1 &\mbox{for all $a \in A(D)$}
\end{array}
\end{equation}
is a compact formulation of the shortest $s,t$-path problem in $(G,c)$.

Combining the shortest path formulations~\eqref{eq:shortestpath} for each edge $e \in E(K[S])$ with a compact formulation for the spanning tree polytope $\Tree(K[S])$ defined
on the complete graph on $S$, we obtain, by
Theorem~\ref{T:dp3},
a compact formulation approximating the Steiner tree problem defined for $(G,c)$.

\begin{Theorem}
Let $G$ be a connected graph, let $c:E(G)\to\R_+$ be a cost function, and let $S \subseteq V(G)$.
Then,
\setcounter{equation}{13}
\begin{align}
&\min \; \sum\limits_{a \in A(D)} \tilde{c}(a) \sum\limits_{e\in E}x^e_a \label{eq:steinerobj}\\[0.8em]
&\begin{array}{r@{\;}l@{\hspace{1.0cm}}r}
x^e(\delta^+(s)) - x^e(\delta^-(s))- \lambda_e &=0 &\mbox{for all $e=\{s,t\}\in E$},\\
x^e(\delta^+(v)) - x^e(\delta^-(v))  &=0 &\mbox{for all $v \in V(D)\setminus \{e\}$},\,e\in E,\\
x^e_a -\lambda_e& \le 0 &\mbox{for all $a \in A(D),\,e \in E$,}\\
x^e_a & \ge 0 &\mbox{for all $a\in A(D),\, e\in E$,}
\end{array}  \label{eq:steiner1} \\[\baselineskip]
\nonumber & (\lambda,\mu) \mbox{ satisfies~\eqref{eq:tree}} \;\;\;(\mbox{or $(\lambda,\nu,\sigma)$ satisfies~\eqref{eq:directedtree}, \eqref{eq:100}}),
\end{align}
where $E:=E(K[S])$,
is a linear programming formulation of size $\MO((|G|+|S|+\langle c \rangle)|S|^2)$ whose optimum is at most two times the optimum
of the Steiner tree problem for $(G,c,S)$, where $|G|:=|V(G)|+|E(G)|$.
\end{Theorem}

\begin{proof}
Clearly, the above linear programs~\eqref{eq:steinerobj} subject to~\eqref{eq:tree}, \eqref{eq:steiner1}
and~\eqref{eq:steinerobj} subject to~\eqref{eq:directedtree}, \eqref{eq:100}, \eqref{eq:steiner1} are compact formulations of the spanning tree approximation
of the Steiner tree problem. So it remains to prove the correctness of
the sizes of these linear programs. We only check the first one.
Clearly, it consists of
$\MO((|A(D)+|S|)|E|)$ variables and
$\MO((|V(D)|+|A(D)|+|S|)|E|)$
inequalities.
Since $V(D)=V(G)$, $|A(D)|=2|E(G)|$, and $|E|=|E(K[S])|=\frac12|S|(|S|-1)$,
the number of variables and inequalities is $\MO((|E(G)+|S|)|S|^2)$
and $\MO((|G|+|S|)|S|^2)$, respectively.
Moreover, each coefficient $c(\{v,w\})$ is counted $2|E|$ times. Hence, the input length of the objective function is $\MO(\langle c \rangle |S|^2)$.
 \end{proof}

\section{A compact formulation for Gomory-Hu trees} \label{Sec:3}

Gomory-Hu trees have many applications in a wide area of graph theory. To mention some examples,
they provide compact representations of minimum $s,t$-cuts for all pairs of nodes $s,t$ of a graph,
they can used to determine minimum $T$-cuts, and they play an important role in the design of minimum cost communication networks.

After some preliminaries, we first consider minimum-require\-ment trees that turns out to be the same as Gomory-Hu trees. As we will see, our approach directly leads to a compact formulation for
Gomory-Hu trees.
 
For any subset of nodes $U \subseteq V(G)$ of a graph $G$,
let $\delta(U)$ be the set of edges of $G$ connecting $U$ and $V\setminus U$.
A {\em cut} in $G$ is an edge set of type $\delta(U)$ for some $\varnothing \neq U \subset V(G)$.
$U$ and $V(G) \setminus U$ are called the {\em shores} of $\delta(U)$.
For two distinct nodes $s$ and $t$ of $G$, an {\em $s,t$-cut} is an edge set of type $\delta(U)$ such that $s \in U$ and $t \in V(G)\setminus U$, or vice versa. We sometimes write $\delta_G(U)$ instead of $\delta(U)$
to indicate that $\delta_G(U)$ induces a cut of $G$.
Given a capacity function $c:E(G)\to\R_+$ and a cut $K \subseteq E(G)$,
the number $c(K)$ is called the {\em capacity} of $K$.
The {\em minimum $s,t$-cut problem} asks for an $s,t$-cut $K \subseteq E(G)$ of minimum capacity.

As it is well known, minimum $s,t$-cut problems in undirected graphs can be represented by
compact linear programs. Let $G$ be a graph, let $c:E(G)\to\R_+$ be a capacity function, and let $s$ and $t$ two distinct nodes of $G$.
For each node $v \in V(G)$, introduce a variable $z_v$, and for each edge $e \in E(G)$, a variable
$x_e$. Then, the model reads: 
\begin{equation} \label{eq:stmincut} \tag{16st}
\begin{array}{lr@{\;}l@{\hspace{1cm}}r}
\multicolumn{4}{l}{\min \; \sum\limits_{e \in E(G)}c(e)x_e}\\[0.8em]
\st & z_s & =0,\\
& z_t & =1,\\
& x_{\{u,v\}}+z_u-z_v & \geq 0,\\
& x_{\{u,v\}}+z_v-z_u & \geq 0 & \mbox{for all } \{u,v\} \in E(G).
\end{array}
\end{equation}

Hu~\cite{Hu74} studied the following problem.
Let $G$ be a graph with ``requirement'' function $c:E(G)\to\R_+$ (say the data volume to be sent between two nodes of a network), and let $K$ be the complete graph on $V(G)$.
A {\em minimum-requirement tree} is a spanning tree $H$ of $K$ minimizing
\setcounter{equation}{16}
\begin{align} \label{eq:mrtree1}
\sum_{e \in E(G)} c(e)\dist_H(e),
\end{align}
where $\dist_H(e)$ denotes the length of the path in $H$ connecting the end nodes of~$e$.

For any edge $f$ of $H$, define $r_H(f)$ to be the capacity ($=$requirement) of the cut
in $G$ determined by the two components of $H-f$.
This cut is called the {\em fundamental cut induced by $f$}.
Then,~\eqref{eq:mrtree1} is equal to
\begin{align} \label{eq:mrtree2}
\sum_{f \in E(H)} r_H(f).
\end{align}

For a fixed spanning tree $H$ of $K$ and a fixed edge $f$ of $H$, the capacity $r_H(f)$ can be easily expressed as the optimal objective value of
the linear program
\begin{equation*}
\begin{array}{lr@{\;}l@{\hspace{1cm}}r}
\multicolumn{4}{l}{\min \; \sum\limits_{e \in E(G)}c(e)x^f_e}\\[0.8em]
\st & z^f_v & =0 &\mbox{for all $v \in U$,}\\
& z^f_v & =1&\mbox{for all $v \in V(G) \setminus U$,}\\
& x^f_{\{u,v\}}+z^f_u-z^f_v & \geq 0,\\
& x^f_{\{u,v\}}+z^f_v-z^f_u & \geq 0 & \mbox{for all } \{u,v\} \in E(G),
\end{array}
\end{equation*}
where $U$ is either of the components of $H-f$. This linear program can be expressed
in terms of the compact spanning tree formulation~\eqref{eq:tree} by identifying $z$- and $\mu$-variables and fixing $\lambda$-variables as follows:
 \begin{equation*}
 \begin{array}{lr@{\;}l@{\hspace{1cm}}r}
\multicolumn{4}{l}{\min \; \sum\limits_{e \in E(G)}c(e)x^f_e}\\[0.8em]
\st & x^f_{\{u,v\}}+\mu_{f,t,u}-\mu_{f,t,v} & \geq 0,\\
& x^f_{\{u,v\}}+\mu_{f,t,v}-\mu_{f,t,u} & \geq 0 & \mbox{for all } \{u,v\} \in E(G),\\
& \lambda_e&=1&\mbox{for all $e \in E(H)$,}\\
&  \lambda_e&=0&\mbox{for all $e \in E(K)\setminus E(H)$,}\\
& \mbox{$(\lambda,\mu)$ satisfies~\eqref{eq:tree}},
\end{array}
\end{equation*}
where $t \in f$ is fixed. Thus, fixing for each edge $f$ of $H$ a node $t_f \in f$,
\eqref{eq:mrtree2} is determined by
\begin{equation*}
 \begin{array}{lr@{\;}l@{\hspace{1cm}}r}
\multicolumn{4}{l}{\min \; \sum\limits_{f\in E(H)}\sum\limits_{e \in E(G)}c(e)x^f_e}\\[0.8em]
\st & x^f_{\{u,v\}}+\mu_{f,t_f,u}-\mu_{f,t_f,v} & \geq 0,\\
& x^f_{\{u,v\}}+\mu_{f,t_f,v}-\mu_{f,t_f,u} & \geq 0 & \mbox{for all $f\in E(H)$, $\{u,v\} \in E(G)$,}\\
& \lambda_e&=1&\mbox{for all $e \in E(H)$,}\\
&  \lambda_e&=0&\mbox{for all $e \in E(K)\setminus E(H)$,}\\
& \mbox{$(\lambda,\mu)$ satisfies~\eqref{eq:tree}}.
\end{array}
\end{equation*}
Hence, it is obvious that the linear program
 \begin{align}
\min & \sum_{f \in E(K)}\sum_{e \in E(G)}c(e)x^f_e \label{eq:ghtreeobj}\\[0.8em]
\nonumber  \st \: & x^f_{\{u,v\}}+\mu_{f,t_f,u}-\mu_{f,t_f,v}  \geq 0,\\
& x^f_{\{u,v\}}+\mu_{f,t_f,v}-\mu_{f,t_f,u}  \geq 0 & \mbox{for all $f\in E(K)$, $\{u,v\} \in E(G)$,}\label{eq:ghtree1}\\
\nonumber& \mbox{$(\lambda,\mu)$ satisfies~\eqref{eq:tree}}.
\end{align}
is a compact formulation for minimum-requirement trees.

\begin{Theorem} \label{T:ght1}
Let $G$ be a graph with capacity function $c:E(G)\to\R_+$, and let $K$ be the complete graph
on $V(G)$. Moreover, let the polyhedron $P$ be the set of all $(\{x^f\}_{f \in E(K)},\lambda,\mu)$
satisfying~\eqref{eq:tree} and~\eqref{eq:ghtree1}. Then, for any
vector $(\{\bar{x}^f\}_{f \in E(K)},\bar{\lambda},\bar{\mu}) \in P$,
$(\{\bar{x}^f\}_{f \in E(K)},\bar{\lambda},\bar{\mu})$ is
a vertex of $P$ minimizing~\eqref{eq:ghtreeobj} if and only if $\bar{\lambda}$ is the characteristic vector of a
minimum-requirement tree $H$ for $(G,c)$,
for each $f \in E(H)$, the vector $(\bar{\mu}_{f,t_f,v})_{v \in V(K)}$ is the characteristic vector
of the component of $H-f$ containing $t_f$, $\bar{x}^f$ is the characteristic vector of the fundamental cut in $G$ induced by $f$, and for each $f \in E(K)\setminus E(T)$, $\mu_{f,u,v}=\bar{x}^f_v=0$ for all
$u \in f,\,v \in V$.
\end{Theorem}

\begin{proof}
$0$ obviously is a lower bound of the objective values of the linear program to be considered, and hence,
always an optimal vertex of $P$ exists.

Any vector $(\{x^f\}_{f \in E(K)},\lambda,\mu) \in P$
is a vertex of $P$ if and only if each component $x^f_{\{u,v\}}$ is chosen to be minimal, that is,
\begin{align*}
x^f_{\{u,v\}}=\max \{\mu_{f,t_f,v} -\mu_{f,t_f,u}, \mu_{f,t_f,u}-\mu_{f,t_f,v}\}
\end{align*}
and $(\lambda,\mu)$ is a vertex of the extension $\Pi$ (constituted by~\eqref{eq:tree}) of the spanning
tree polytope $\Tree(K)$. So by Lemma~\ref{L:tree}, $(\{x^f\}_{f \in E(K)},\lambda,\mu)$ is vertex
if and only if satisfies the conditions mentioned in Theorem~\ref{T:ght1}.

To conclude, for any vertex $(\{x^f\}_{f \in E(K)},\lambda,\mu)$ of $P$ and any spanning tree $H$ of $K$,
$\lambda=\chi^{E(H)}$ implies $\sum_{e\in E(G)}c(e)x^f_e=r_H(f)$ if $f \in E(H)$
and equals $0$ otherwise. Thus, $(\{x^f\}_{f \in E(K)},\lambda,\mu)$ minimizes~\eqref{eq:ghtreeobj} if and only if
$H$ is a minimum-requirement tree.
\end{proof}

Next, we turn to Gomory-Hu trees. Given a graph $G$ with capacity function $c:E(G)\to\R_+$,
a {\em Gomory-Hu tree} for $(G,c)$ is a tree $H$ with $V(H)=V(G)$
such that for each edge $e=\{s,t\}$ of $H$,
the cut determined by the two components of $H-e$ is a minimum $s,t$-cut of $G$.

In the {\em all-pairs minimum cut problem}, one wants to find a  minimum $s,t$-cut for
all pairs $s,t \in V(G)$.
This problem can be represented in a compact form by Gomory-Hu trees.
Let $H$ be a Gomory-Hu tree for $(G,c)$.
For $e=\{s,t\} \in E(K)$, define $r(\{s,t\})$ as the capacity of a minimum $s,t$-cut of $G$.
Then,
\begin{align} \label{eq:ghprop}
r(\{u,v\})&=\min_{e\in E(P_{uv})} r(e) &\mbox{for all $\{u,v\} \in E(K)$,}
\end{align}
where $P_{uv}$ is the unique path in $H$ connecting $u$ and $v$.

Hu~\cite{Hu74} showed, based on an earlier result of Adolphson and Hu~\cite{AH73},
that Gomory-Hu trees are minimum-requirement trees.
Considering the proofs of the results in~\cite{AH73,Hu74} it turns out that, conversely,
minimum-requirement trees are Gomory-Hu trees. Below we restate Hu's result
taking into consideration this equivalence.

\begin{Theorem} \label{T:ght2}
Let $G$ be a graph with capacity function $c:E\to\R_+$, and let $K$ be the complete graph on $V(G)$.
Then, $H \subseteq K$ is a Gomory-Hu tree for $(G,c)$ if and only if $H$ is a minimum-requirement tree
for $(G,c)$.
\end{Theorem}

\begin{proof}
Let $H$ and $H'$ be a Gomory-Hu and a minimum-requirement tree for $(G,c)$, respectively.
By definition of $H$, for each edge $f=\{s,t\} \in E(H)$ and each edge $f'$ on the $s,t$-path in $H'$
one has
\begin{align}
r_{H'}(f') \geq r_H(f), \label{eq:eq1}
\end{align}
as $r_H(f)$ is the capacity of a minimum $s,t$-cut and the components of $H'-f'$ determine
an $s,t$-cut.
There exists a bijection $\varphi: E(H)\to E(H')$ such that
for each $f=\{s,t\} \in E(H)$, $\varphi(f)$ is an edge on the $s,t$-path in $H'$, since $H$ and $H'$ are spanning trees. For details, see~\cite{AH73} or
\cite[Section~15.4a]{Schrijver03}.
So~\eqref{eq:eq1} implies
\begin{align} \label{eq:mrtree3}
\sum_{f' \in E(H')} r_{H'}(f') = \sum_{f \in E(H)} r_{H'}(\varphi(f)) \geq \sum_{f \in E(H)} r_H(f).
\end{align}
Since~\eqref{eq:mrtree1} and~\eqref{eq:mrtree2} are equal, $H'$ minimizes~\eqref{eq:mrtree2},
and hence all sums in~\eqref{eq:mrtree3} are equal.
This implies that $H$ is a minimum-requirement tree.
Moreover, because of~\eqref{eq:eq1}, $r_{H'}(f'))=r_H(\varphi^{-1}(f'))$ for all $f' \in E(H')$.
Hence, $H'$ is a Gomory-Hu tree.
\end{proof}

\begin{Corollary} \label{C:ght2}
Let $G$ be a graph, and let $c:E(G)\to\R_+$ be a capacity function.
The linear program~\eqref{eq:ghtreeobj} subject to~\eqref{eq:tree} and~\eqref{eq:ghtree1} is a compact formulation for
Gomory-Hu trees for $(G,c)$.
\end{Corollary}

The linear program~\eqref{eq:ghtreeobj} subject to~\eqref{eq:tree} and~\eqref{eq:ghtree1} does not exactly represent the two-level approach described in Section~\ref{Sec:2},
but it is in the same spirit. In this context, recall that~\eqref{eq:ghprop} implies that $H$ is a maximum spanning tree for $(K,r)$. However, not every maximum spanning tree for $(K,r)$ is a Gomory-Hu tree, see Schrijver~\cite[Section 15.4]{Schrijver03}. Without this
restriction, we only had to couple the min-cut formulations~\eqref{eq:stmincut} with a compact representation of the convex hull of the characteristic vectors of the edge complements of spanning trees (called {\em co-trees}), obtained, for instance, by a simple modification of system~\eqref{eq:tree}. (In this context, note that Theorem~\ref{T:dp3} presumes to couple either maximization or minimization problems.)  The restriction that not every maximum spanning tree for $(K,r)$ is a Gomory-Hu tree leads then to the idea to identify $\mu$- and $z$-variables
as presented above.

In what follows, we describe some consequences of Theorem~\ref{T:ght2}.

\subsection{The Gomory-Hu tree polytope} \label{Sec:3.1}

Let $G$ be a graph with capacity function $c:E(G)\to\R^+$.
The {\em Gomory-Hu tree polytope} $\GHTree(G,c)$ is the convex hull of the characteristic vectors of
Gomory-Hu trees for $(G,c)$. Using standard linear programming techniques, a compact formulation for $\GHTree(G,c)$
can be derived from the linear program~\eqref{eq:ghtreeobj} subject to~\eqref{eq:tree},
\eqref{eq:ghtree1}.

Given a pair
$$(P)\;\; \min \{w^Tx\,|\,Ax \geq b\}\quad\quad\quad (D) \;\; \max \{y^Tb\,|\,y^TA=w^T,\,y\geq 0\}$$
of dual linear programs. If both (P) and (D) have feasible solutions, then
$$Q:= \{(x,y)\,|\, Ax \geq b,\,y^TA=w^T, w^Tx - y^Tb=0,\,y\ge 0\}$$
is a polyhedron that consists of all vectors $(x,y)$ such that $x$ and $y$ are optimal for (P)
and (D), respectively.

If (P) is the linear program~\eqref{eq:ghtreeobj} subject to~\eqref{eq:tree},~\eqref{eq:ghtree1}, then we conclude from Theorem~\ref{T:ght2} that $\GHTree(G,c)$ is the projection of
$Q$ onto the $\lambda$-space.

\begin{Corollary} \label{C:3.4}
Let $G$ be a graph, and let $c:E(G)\to\R_+$ be a capacity function.
Then, there exists an $O(|V(G)|^4+ |V(G)|^2\langle c\rangle)$ extended formulation of $\GHTree(G,c)$, where $\langle c \rangle$ denotes the input length of $c$.
\end{Corollary}

\subsection{The minimum $T$-cut problem}

Let $G$ be a graph with capacities on the edges $c:E(G) \to \R_+$,
and let $T$ be a subset of the nodes of $G$ of even size.
A {\em $T$-cut} is a cut $\delta(U)$ of $G$ such that $|T \cap U|$ is odd.
In the {\em minimum $T$-cut problem} one wants to find a $T$-cut $\delta(U)$ of $G$ minimizing
$c(\delta(U))$. The polyhedral counterpart to the minimum $T$-cut problem is the {\em $T$-cut polyhedron},
defined as the set of all $x \in \R^{E(G)}$ for which there exists a convex combination $y$ of the characteristic vectors of $T$-cuts such that $x \geq y$.

The minimum $T$-cut problem can be solved in polynomial time. Various algorithms are available  and, among them, the famous algorithm of Padberg and Rao~\cite{PR82} that
computes a Gomory-Hu tree for $(G,c)$,
and selects among the fundamental cuts a $T$-cut of minimum capacity.
This $T$-cut has minimum capacity among all $T$-cuts of $G$.

This algorithm can be used as as basis to derive a compact formulation for the minimum $T$-cut problem.
However, the inequalities of our formulation will depend on the input $c$, and as a consequence, this is only a formulation for the minimum $T$-cut problem but not for the $T$-cut polyhedron.

Let $K$ be the complete graph on $V(G)$, let $H$ be a Gomory-Hu tree for $(G,c)$, and let $\lambda$ its characteristic vector.
Recall that for $f \in E(H)$, the capacity of the fundamental cut induced by $f$ is the optimal objective value of the linear program
\setcounter{equation}{24}
\begin{align}
 \min \;& \sum_{e \in E(G)}c(e)x^f_e \tag{24f} \label{eq:24f}\\[0.8em]
\st\; &\lambda_e=1&\mbox{for all $e \in E(H)$,} \label{eq:25}\\
& \lambda_e=0&\mbox{for all $e \in E(K)\setminus E(H)$,} \label{eq:26}\\
&\begin{array}{l}
x^f_{\{u,v\}}+\mu_{f,t_f,u}-\mu_{f,t_f,v}  \geq 0,\\
x^f_{\{u,v\}}+\mu_{f,t_f,v}-\mu_{f,t_f,u}  \geq 0
\end{array} & \mbox{for all } \{u,v\} \in E(G), \label{eq:27f} \tag{27f} \\
\nonumber& \mbox{$(\lambda,\mu)$ satisfies~\eqref{eq:tree}},
\end{align}
where $t_f \in f$ is fixed, while for $f \in E(K) \setminus E(H)$ the optimal value is zero.
Introduce for each $f \in E(K)$ a (binary) variable $\nu_f$ such that
$\nu_f=1$ if and only if $f \in E(H)$ and the fundamental cut induced by $f$ is a $T$-cut.
Denote the edge set corresponding to $T$-cuts by $F$.
Then, the following mathematical program is a disjunctive programming approach for finding a minimum $T$-cut among the fundamental cuts of $H$:
\setcounter{equation}{26}
\begin{align}
 \min \;& \sum_{f \in E(K)} \sum_{e \in E(G)}c(e)y^f_e \tag{24} \label{eq:24}\\[0.8em]
\nonumber \st \;&\mbox{$(\lambda,\mu)$ satisfies~\eqref{eq:tree}}, \eqref{eq:25}, \eqref{eq:26},\\
&\mbox{$(\{x^f\}_{f \in E(K)},\mu)$ satisfies~\eqref{eq:27f}} & \mbox{for all $f \in E(K)$}, \label{eq:27}\\
&\nu_f = \left\{ \begin{array}{cl}
1&\mbox{if $f \in F$}\\
0& \mbox{otherwise,}
\end{array} \right. \label{eq:28}\\
&\sum_{f \in E(K)}\vartheta_f = 1,\; 0 \leq \vartheta \leq \nu, \label{eq:29}\\
& y^f_e \geq x^f_e+\vartheta_f-1,\; y^f_e \geq 0 & \mbox{for all $e \in E(G), f \in E(K)$.} \label{eq:30}
\end{align}

To see the correctness of this formulation, let $(\{x^f,y^f\}_{f \in E(K)}, \lambda,\mu,\nu,\vartheta)$ be a vertex of the polyhedron determined by the inequalities of this formulation.
Recall that $\lambda$ is the characteristic vector of $H$. Moreover, if $f=\{s,t\} \in E(H)$,
$\mu_{f,s,\star}, \mu_{f,t,\star}$ are the characteristic vectors of the components of $H-f$, and $x^f$ represents the associated fundamental cut. Otherwise, that is, in case $\lambda_f=0$,
$x^f=\mu_{f,s,\star}=\mu_{f,t,\star}=0$. Moreover, since  $(\{x^f,y^f\}_{f \in E(K)}, \lambda,\mu,\nu,\vartheta)$ is a vertex, inequalities~\eqref{eq:28}, \eqref{eq:29} imply that $\vartheta$ is a unit vector with $\vartheta_g=1$ for some $g \in F$.
For the same reason, $(\{x^f,y^f\}_{f \in E(K)}, \lambda,\mu,\nu,\vartheta)$ satisfies at least one of the two equations in~\eqref{eq:30} at equality for each pair of edges $e \in E(G), f \in E(K)$. Hence,
$y^f=0$ for all $f \in E(K) \setminus \{g\}$, while $y^g$ is the characteristic vector of a $T$-cut.

In an instance of the minimum $T$-cut problem neither $H$ nor $F$ are part of the input.
Therefore, to derive a compact linear formulation from the above mathematical program
we have to remove the fixing equations~\eqref{eq:24}, \eqref{eq:25} and to replace the fixing equations~\eqref{eq:28} by an appropriate system of inequalities.

In a first step, let us assume that
$H$ but only $H$ is part of the input. One way to determine $\nu$ is as follows.
Fix any node $r \in T$, and for each $s \in T \setminus \{r\}$, denote by $P_s$ the unique $r,s$-path in $H$. Then, for any edge $f$ of $H$, $f \in F$ if and only if the number of paths $P_s$ using $f$ is odd. This, in turn, is the case if and only if the symmetric difference of the paths $P_s$ contain $f$.

Let $s \in T \setminus \{r\}$. Introducing for each edge $e \in E(K)$ a variable $\pi^s_e$, one easily checks that the characteristic vector of $P_s$ is, for instance, determined by the system
\begin{align}
& \begin{array}{r@{\;}l@{\hspace{1cm}}r}
y^s(\delta(r))=y^s(\delta(s)) &=1,\\
y^s(\delta(v) \setminus{e})-y^s_e & \geq 0 & \mbox{for all $v \in V(K)$, $e \in \delta(v)$,}\\
0 \leq y^s & \leq  \lambda.
\end{array}  \label{eq:31s} \tag{31s}
\end{align}

Next, we have to express the symmetric difference of the paths $P_s$ by linear inequalities.
As it is well-known, the characteristic vector of the symmetric difference $X \Delta Y:=(X \cup Y) \setminus (X \cap Y)$ of any
sets $X,Y \subseteq E(K)$ is determined by the system
\begin{align*}
\alpha_f & \leq \chi^X_f+\chi^Y_f & \mbox{for all $f \in E(K)$},\\
\alpha_f & \geq \chi^X_f-\chi^Y_f & \mbox{for all $f \in E(K)$},\\
\alpha_f & \geq \chi^Y_f+\chi^X_f & \mbox{for all $f \in E(K)$},\\
\alpha_f & \leq 2-\chi^X_f-\chi^Y_f & \mbox{for all $f \in E(K)$}.
\end{align*}
Since the $\Delta$-operator is associative, the symmetric difference of the paths $P_s$ can be determined sequentially. For this, let $(s_0,s_1,\ldots,s_k)$ be any order of the nodes in $T \setminus \{r\}$. Then, defining $D_0:= P_{s_0}$ and $D_i:=D_{i-1} \Delta P_{s_i}$ for $i=1,\ldots,k$, it follows that $F=D_k$. Thus, introducing for each step $i \in \{0,1,\ldots,k\}$, an edge variable vector $\alpha^i$ to represent $D_i$, we see that equations~\eqref{eq:28} can be replaced by the system
\begin{align}
\mbox{$(y^{s_i},\lambda)$ satisfies}\;&\eqref{eq:31s} & \mbox{for $s=s_i$, $i=0,1,\ldots,k$},\\
\alpha^0 - y^{s_0}&=0\\
 -\alpha^i +\alpha^{i-1}+y^{s_i} &\geq 0 & \mbox{for $i=1,\ldots,k$},\\
\alpha^i -\alpha^{i-1}+y^{s_i} &\geq 0 & \mbox{for $i=1,\ldots,k$},\\
\alpha^i +\alpha^{i-1}-y^{s_i} &\geq 0 & \mbox{for $i=1,\ldots,k$},\\
\alpha^i + \alpha^{i-1}+y^{s_i} & \leq 2 & \mbox{for $i=1,\ldots,k$},\\
\alpha^k - \nu &=0. \label{eq:37}
\end{align}

In the remainder of this section we consider the linear program~\eqref{eq:24} subject to~\eqref{eq:tree}, \eqref{eq:25}-\eqref{eq:27}, \eqref{eq:29}-\eqref{eq:37}. Removing the fixing variables~\eqref{eq:25}, \eqref{eq:26}, we obtain, of course, a compact formulation for a relaxation of the minimum $T$-cut problem. This implies that the system~\eqref{eq:tree}, \eqref{eq:27}, \eqref{eq:29}-\eqref{eq:37} characterizes a polyhedron that contains the $T$-cut polyhedron, but we conjecture that this system does not determine the $T$-cut polyhedron.
However, using standard techniques, a compact formulation for a particular instance of the minimum $T$-cut problem can be derived as follows.
By Corollary~\ref{C:3.4}, there exists a compact extension $Q$ of the Gomory-Hu polytope $\GHTree(G,c)$. We assume that $Q$ is given in $\lambda,\mu,\eta$- space.
Let $(f_1,f_2,\ldots,f_m)$ be any ordering of the edges of $K$.
Then, the unique optimal solution of the linear program
\begin{align*}
\min \;& \sum_{j=1}^m 2^{j-1} \lambda_{e_j}\\
\st\;& \mbox{$(\lambda,\mu,\eta) \in Q$}
\end{align*}
is the characteristic vector of the minimal lexicographic Gomory-Hu tree with respect to the ordering. $Q$ is determined by a compact system of the form
$$A \begin{pmatrix}
\lambda\\
\mu\\
\eta
\end{pmatrix} \geq b,
$$
and by the construction in Subsection~\ref{Sec:3.1}, we may assume that this system contains~\eqref{eq:tree} as subsystem.
The characteristic vector of the minimal lexicographic Gomory-Hu tree is the projection of
the polytope determined by
\begin{equation} \label{eq:38}
A \begin{pmatrix}
\lambda\\
\mu\\
\eta
\end{pmatrix} \geq b, \pi^TA = d^T, d^T
\begin{pmatrix}
\lambda\\
\mu\\
\eta
\end{pmatrix} - \pi^Tb =0, \pi \geq 0
\end{equation}
onto the $\lambda$-space,
where $d^T:= (1,2,4,\ldots,2^{m-1},0,0,\ldots,0)$, and hence,
the linear program~\eqref{eq:24} subject to~\eqref{eq:27}, \eqref{eq:29}-\eqref{eq:38}
is a compact formulation for the minimum $T$-cut problem.
However, we note that~\eqref{eq:38} contains inequalities some of whose coefficients have input length $|E(K)|$, and more important, some depend on $c$,
and thus this is not a formulation for the $T$-cut polyhedron.

\section{Conclusion} \label{Sec:4}
Motivated by a well-known result of Balas on disjunctive programming,
we studied extended formulations in connection with a simple two-level optimization scheme
that also can be derived as a special case from the more general framework {\em branched polyhedral systems} of Kaibel and Loos~\cite{KL10}. Using this scheme, we gave compact formulations for the spanning tree approximation
of the Steiner tree problem and for Gomory-Hu trees. Using the Gomory-Hu tree formulation, we also derived a compact formulation for the minimum $T$-cut problem whose inequalities, however, depend on the objective function.

A very interesting question in this context is if the compact formulation for the minimum $T$-cut problem can be modified in such a way that it extends to a compact formulation for the $T$-cut polyhedron. However, answering this question seems to be a complicated undertaking,
since the separation problem for the perfect matching polytope defined on a graph $G$ can be modeled as linear optimization problem over the $V(G)$-cut polyhedron. Such a formulation would also imply a compact formulation for the perfect matching polytope -- a long standing unsolved task.

A powerful theory on coupling extended formulations has been developed, especially due to the contribution of Kaibel and Loos~\cite{KL10}. In future research, we are interested to find further examples, where the coupling of extended formulations yields to compact formulations for polynomially solvable combinatorial optimization problems. 

\subsection*{Acknowledgments}
I am very grateful to Laurence Wolsey for useful discussions on compact formulations of Gomory-Hu trees and for his helpful suggestions to  improve the readability of this paper.

%\bibliographystyle{../mod_siam}
%\bibliography{dp}

\begin{thebibliography}{10}

\bibitem{AH73}
\textsc{D.~Adolphson and T.~Hu}, \emph{{Optimal linear ordering}}, SIAM J.
  appl. Math. \textbf{25} (1973), pp.~403--423.

\bibitem{Balas79}
\textsc{E.~Balas}, \emph{{Disjunctive programming}}, Ann. Discrete Math.
  \textbf{5} (1979), pp.~3--51.

\bibitem{Balas05}
\textsc{E.~Balas}, \emph{{Projection, lifting and extended formulation integer
  and combinatorial optimization}}, Ann. Oper. Res. \textbf{140} (2005),
  pp.~125--161.

\bibitem{Balas10}
\textsc{E.~Balas}, \emph{{Disjunctive Programming}}.
\newblock {In J\"unger, M. (ed.) et al.: 50 Years of Integer Programming
  1958-2008. From the Early Years to the State-of-the-Art, Springer, Berlin,
  283-340}, 2010.

\bibitem{CCZ09}
\textsc{M.~Conforti, G.~Cornu\'ejols, and G.~Zambelli}, \emph{Extended
  formulations in combinatorial optimization}, tech. report, Universit\`a di
  Padova, 2009.

\bibitem{GP68}
\textsc{E.~Gilbert and H.~Pollak}, \emph{{Steiner minimal trees.}}, SIAM J.
  Appl. Math. \textbf{16} (1968), pp.~1--29.

\bibitem{Hu74}
\textsc{T.~Hu}, \emph{{Optimum communication spanning trees}}, SIAM J. Comput.
  \textbf{3} (1974), pp.~188--195.

\bibitem{KL10}
\textsc{V.~Kaibel and A.~Loos}, \emph{{Branched polyhedral systems}}.
\newblock {In Eisenbrand, F. and Shepherd, B. (eds.): Integer programming and
  combinatorial optimization. Proceedings of IPCO XIV, Ithaca NY. Lecture Notes
  in Computer Science {\bf 6080}, Springer, 283-340}, 2010.

\bibitem{Maculan86}
\textsc{N.~Maculan}, \emph{{A new linear programming formulation for the
  shortest s-directed spanning tree problem.}}, J. Comb. Inf. Syst. Sci.
  \textbf{11}, no.~2-4 (1986), pp.~53--56.

\bibitem{KMartin91}
\textsc{R.~Martin}, \emph{{Using separation algorithms to generate mixed
  integer model reformulations}}, Oper. Res. Lett. \textbf{10}, no.~3 (1991),
  pp.~119--128.

\bibitem{MRC88}
\textsc{R.~Martin, R.~L. Rardin, and B.~A. Campbell}, \emph{{Polyhedral
  characterization of discrete dynamic programming.}}, Oper. Res. \textbf{38},
  no.~1 (1990), pp.~127--138.

\bibitem{PR82}
\textsc{M.~W. Padberg and M.~Rao}, \emph{{Odd minimum cut-sets and
  b-matchings.}}, Math. Oper. Res. \textbf{7} (1982), pp.~67--80.

\bibitem{PW93}
\textsc{Y.~Pochet and L.~A. Wolsey}, \emph{{Lot-sizing with constant batches:
  Formulation and valid inequalities.}}, Math. Oper. Res. \textbf{18}, no.~4
  (1993), pp.~767--785.

\bibitem{Schrijver03}
\textsc{A.~Schrijver}, \emph{{Combinatorial optimization. Polyhedra and
  efficiency}}, {Algorithms and Combinatorics 24. Berlin: Springer}, 2003.

\bibitem{Wong84}
\textsc{R.~T. Wong}, \emph{{A dual ascent approach for Steiner tree problems on
  a directed graph.}}, Math. Program. \textbf{28} (1984), pp.~271--287.

\bibitem{Yannakakis91}
\textsc{M.~Yannakakis}, \emph{{Expressing combinatorial optimization problems
  by linear programs}}, J. Comput. Syst. Sci. \textbf{43}, no.~3 (1991),
  pp.~441--466.

\end{thebibliography}

\end{document}